\documentclass{article}
\usepackage[utf8]{inputenc}
\usepackage{settings}
\usepackage[title]{appendix}
\usepackage{sectsty}
\usepackage{lipsum}
\newcommand{\Addresses}{{
		\bigskip
		
		Zhenlin Ran, \textsc{University of Newcastle, Australia}\par\nopagebreak
		\textit{E-mail address:} \texttt{zhenlin.ran@newcastle.edu.au}
}}

\title{\Large HEIGHTS AND SINGULAR MODULI OF DRINFELD MODULES}
\author{Zhenlin Ran}
\date{}

\allsectionsfont{\normalfont\scshape\centering}

\begin{document}

\maketitle

\begin{abstract}
	Let $q$ be an odd number and $q>5$, and $\fq$ be a finite field of $q$ elements. We prove that at most finitely many singular moduli of rank 2 $\fq[t]$-Drinfeld modules are algebraic units. In particular, we develop some techniques of heights of Drinfeld modules to approach it.
\end{abstract}

\tableofcontents

\section{Introduction}
A \textit{singular modulus} is the $j$-invariant of a CM elliptic curve over the complex numbers. It is well-known that a singular modulus is an algebraic integer. In 2015, Habegger proved the following theorem \cite{ha14}:

\begin{thm}
At most finitely many singular moduli are algebraic units.
\end{thm}

His proof employs many classical tools from Diophantine geometry. The main idea of his proof is to bound the Weil height of a unitary singular modulus. In particular, assuming the singular modulus is a unit, an upper bound of its Weil height could be obtained by applying an equidistribution theorem from Clozel and Ullmo \cite[Section 2.3]{cu04}. This is the biggest difficulty of the entire proof as the number of singular moduli with big Galois orbits is hard to control.\\

In this paper, we consider a function field analogue and prove the same result for Drinfeld $\fq[t]$-modules.\\

Let $C$ be a smooth, projective and geometrically irreducible curve over a finite field $\fq$ which has $q$ elements. Fix a closed point $\infty\in C$ and let $A$ be the ring of functions on $C$ regular outside $\infty$. For instance, if $C$ is a projective line, then $A=\fq[t]$ for some $t$ transcendental over $\fq$. A \textit{Drinfeld $A$-module} over some scheme $S$ over $A$ is a pair $(\G_{a,\Lcal},\phi)$ such that $\G_{a,\Lcal}$ is a line bundle over $S$ given by an invertible sheaf $\Lcal$ and $\phi$ is a ring homomorphism from $A$ to $\End(\G_{a,\Lcal})$ with some extra conditions (cf. Definition \ref{dfdm}). It is well-known that Drinfeld $A$-modules of rank 2 are the analogue of elliptic curves. Most of the concepts and results of elliptic curves could be found for Drinfeld $A$-modules of rank 2. For example, we can define singular modulus of Drinfeld $\fq[t]$-modules in the same way as elliptic curves. The main theorem of this paper is:

\begin{thm}\label{mthm}
	Let $q$ be odd and $q>5$. There are at most finitely many singular moduli of rank 2 Drinfeld $\fq[t]$-modules that are algebraic units.
\end{thm}

The strategy of proving our main theorem follows that of Habbeger. As we pointed out earlier, the key tool that Habegger uses to control the number of Galois orbits of quadratic imaginary numbers close to the roots of the $j$-function is an equidistribution theorem from Clozel and Ullmo, which enables him to obtain an upper bound for the Weil height of a unitary singular modulus. However, this idea does not work well for our case since to the best of the author's knowledge, there are not any equidistribution results for Drinfeld $A$-modules like the one of Clozel and Ullmo. Instead, our idea to address this issue is to study the arithmetic of quadratic imaginary points. Though our method for the case of Drinfeld $\fq[t]$-modules is somehow elementary, we can show that there is at most one quadratic imaginary point in a certain small neighbourhood of a root of the $j$-function (cf. Proposition \ref{ptcp}). Thus, we could also obtain an upper bound for the Weil height of unitary singular moduli of Drinfeld $\fq[t]$-modules.\\

On the other hand, Habegger also gives a lower bound for the Weil height of singular moduli that grows faster than the upper bound he obtained. Many tools for the case of elliptic curves were already known while the analogues for Drinfeld $A$-modules are not available. We follow Habegger's strategy and prove some analogous results for the case of Drinfeld modules, which will lead us to a lower bound for the Weil height of singular moduli for Drinfeld $\fq[t]$-modules. That is:

\begin{thm}
	Let $J$ be a singular modulus of rank 2 Drinfeld $\fq[t]$-module with corresponding discriminant $\delta$ with conductor $f_0$. There exists some computable constant $C_q$ with respect to $q$ such that
	$$h(J)\geq (q^2-1)\left( \frac{1}{2}-\frac{1}{\sqrt{q}+1}\right) \log\sqrt{|\delta|}+\left( \frac{1}{2}+\frac{1}{\sqrt{q}+1}\right)\log|f_0|-\frac{9}{4}\log\log |f_0|-C_q.$$
\end{thm}

In Drinfeld $A$-modules, the analogue of Faltings height is Taguchi height which was introduced by Taguchi in \cite{tag1991} for the case of finite characteristic and in \cite{tag1993} for the case of generic characteristic. In particular, we obtain an analogous result of Nakkajima and Taguchi for Drinfeld $\fq[t]$-modules, which gives an explicit description for the variation of the Taguchi heights of rank 2 Drinfeld $\fq[t]$-modules under isogeny, where one Drinfeld $\fq[t]$-module of rank $2$ has CM by arbitrary order and the other one has CM by the maximal order. From this point on, we can obtain the variation of the graded heights of the corresponding Drinfeld $\fq[t]$-modules. As the graded height of Drinfeld $A$-modules is the generalization of the Weil height of $j$-invariants for Drinfeld $A$-modules, we thus obtain the above lower bound for the Weil height of singular moduli by applying a theorem of Wei where he proves the Colmez conjecture for Drinfeld $A$-modules.\\

More recently Bilu, Habegger and K\"uhne proved the stronger result that there are no singular moduli that are algebraic units in the case of elliptic curves \cite{bhk2020}. Their approach is the same as Habegger's while the method in \cite{bhk2020} is effective. What makes Habegger's proof ineffective is the equidistribution theorem he applied. Compared to this work, a totally different approach using Gross-Zagier, Gross-Kohnen-Zagier and their generalizations was given by Li \cite{li2021}. Li could deduce the same result of \cite{bhk2020} as a special case of \cite[Theorem 1.1]{li2021}.\\

This paper is organized as follows:

\begin{itemize}
\item In Section 2, we recall the concept of Drinfeld $A$-modules over arbitrary $A$-schemes and the minimal models of Drinfeld $A$-modules. In particular, we study some properties of the minimal models.
\item In Section 3, three different types of heights are introduced and their relationships are studied. In particular, we study the variation of heights of Drinfeld $A$-modules under isogenies.
\item In Section 4, we study the number of quadratic imaginary points near the root of the $j$-function of Drinfeld $A$-modules.
\item In Section 5, we bound the Weil height of $j$-invariants and prove the main theorem.
\item The appendix is written to give some results related to CM Drinfeld $A$-modules that could not be found in some common literature. In particular, this section is devoted to the proof of Proposition \ref{abdf}.
\end{itemize}

\noindent \textbf{Acknowledgements}\quad The author is very grateful to his advisor Florian Breuer for advising him such an interesting project and for being supportive in all aspects throughout his PhD study. He would also thank Marc Hindry and Urs Hartl for inviting him to visit Paris and M\"unster respectively where the author had delightful and helpful discussions with them and learnt mathematics from them. He also appreciate Fabien Pazuki and Fu-Tsun Wei for helpful discussions. Many thanks go to Philipp Habegger for his suggestion of working on Lemma \ref{etad}.

\section{Minimal models of Drinfeld $A$-modules}

In this section, we review the definition of Drinfeld $A$-modules and the associated minimal models. In particular, we prove an analogue of a result of N\'eron models of abelian varieties.\\

Let $C$ be a smooth, projective and geometrically irreducible curve over $\fq$, where $\fq$ is a finite field with $q$ elements. Let $\infty\in C$ be a closed point. We set $A:=\Gamma(C\backslash\{\infty\}, \Ocal_C)$ to be the ring of functions regular outside $\infty$. We fix $k$ to be the field of fractions of $A$. Let $M_k$ denote the set of all places of $k$. 
To each place $v\in M_k$ we associate an absolute value $|\cdot|_v$ as follows: $$|x|_v=q^{-\deg(v)v(x)},\ \forall x\in k.$$ 
Let $k_{\infty}$ denote the completion of $k$ with respect to $\infty$ and $\Cinf$ denote the completion of an algebraic closure of $k_{\infty}.$\\

Throughout this paper, we denote by $\log$ the logarithm function of base $q$ and assume that \textit{$q$ is odd}.

\begin{defn}\label{dfdm}
Let $S$ be a scheme over $\spec(A)$ with structure morphism $i: S\ra \spec(A)$. A \textit{Drinfeld $A$-module} over $S$ is a pair $\textbf{E}=(\G_{a,\Lcal},\phi)$, where $\Lcal$ is an invertible sheaf over $S$ and $\phi$ is a homomorphism from $A$ to $\End(\G_{a,\Lcal})$ such that 
\begin{enumerate}[(1)]
\item $\partial \circ \phi=i^{\#}$, where $i^{\#}:A\ra \Ocal_S(S)$ and $\partial$ is the natural homomorphism of taking the constant term of $\End(\G_{a,\Lcal})$.;
\item for any $0\neq a\in A$, the morphism $\phi(a)$ is finite;
\item at any point of $S$ the degree of $\phi(a)$ is $>1$ for some $a\in A$.
\end{enumerate}
\end{defn}

\begin{rem}
\begin{enumerate}[(1)]
\item If $\varphi\in \End(\G_{a,\Lcal})$ then $\varphi=\sum_{n\geq 0}a_n\tau_p^n$, where $a_n\in \Gamma(S,\Lcal^{1-p^n})$ and $\tau_p$ is the relative $p$-Frobenius. So $D\circ \phi$ lands in $\Ocal_S(S)$.
\item We denote the subalgebra of $\fq$-linear endomorphisms of $\End(\G_{a,\Lcal})$ by $\End_q(\G_{a,\Lcal})$. The elements of $\End_q(\G_{a,\Lcal})$ are of the form $\sum_n a_n\tau_q^n$ for $a_n\in \Gamma(S,\Lcal^{1-q^n})$.
\item For any $\varphi=\sum_{n\geq 0}a_n\tau_p^n\in \End(\G_{a,\Lcal})$, the sum is locally finite \cite[Remark 1.2.4]{leh09}.
\item If $S$ is connected, then there exists an integer $r>0$ such that $\deg(\phi(a))=|a|_{\infty}^r$ \cite[Proposition 2.2.2]{leh09}. The integer $r$ is called the \textit{rank} of the Drinfeld module $\textbf{E}$.
\item A \textit{homomorphism} (resp. \textit{isogeny}) from $(\G_{a,\Lcal},\phi)$ to $(\G_{a,\mathcal{M}},\psi)$ is a (resp. finite) homomorphism $f:\G_{a,\Lcal} \ra \G_{a,\mathcal{M}}$ such that $f\circ \phi(a)=\psi(a)\circ f$ for all $a\in A$.
\end{enumerate}
\end{rem}

We abbreviate $\phi(a)$ as $\phi_a$ for $a\in A$. If $S$ is the spectrum of a field, then the line bundle on $S$ is unique up to isomorphism in which case we specify a Drinfeld module over $S$ only by $\phi$, which we implicitly take a trivialization and regard such Drinfeld $A$-modules as a ring homomorphism $\phi$ such that (1)-(3) of Definition \ref{dfdm} hold. In this paper, we assume for any Drinfeld module $(\G_{a,\Lcal},\phi)$, $\phi$ is $q$-linear and $\tau:=\tau_q$.

\begin{defn}\label{mimd} (Taguchi)
Let $S$ be an integral normal scheme of finite type over $A$ with function field $F$. Let $\phi$ be a Drinfeld $A$-module over $F$. A \textit{model} $\mathscr{M}=(\G_{a,\Lcal},\varphi,f)$ of $\phi$ over $S$ is an $A$-module scheme $\textbf{E}=(\G_{a,\Lcal},\varphi)$ over $S$ such that $f:\textbf{E}\times_S \spec(F)\ra \phi$ is an isomorphism of Drinfeld modules over $F$. A model $\mathscr{M}$ of $\phi$ over $S$ is \textit{minimal} if given any other model $\mathscr{N}=(\G_{a,\Lcal'},\varphi',f')$, there exists a unique homomorphism $\mathscr{N}\ra \mathscr{M}$ which induces an isomorphism on the generic fibre compatible with $f$ and $f'$.
\end{defn}

\begin{prop}\cite[Proposition 2.2]{tag1993}
Let $S$ and $F$ be as in Definition \ref{mimd}, and we further assume $S$ is a scheme on which the two concepts of Weil divisors and Cartier divisors coincide. Then there exists a minimal model over $S$ of $\phi$.
\end{prop}

\begin{rem}
\begin{enumerate}[(1)]
\item If $\phi$ has a minimal model, then the minimal model is unique up to isomophism;
\item By checking on fibres we see that every model of $\phi$ over $S$ is smooth over $S$.
\item By \cite[Remark (4), p.299]{tag1993}, each model is isomorphic to a standard one. To avoid confusion with the term standard Drinfeld module, we call the standard model from \cite[Remark (4), p. 299]{tag1993} the \textit{normalized model}. That is, a \textit{normalized model} is a model whose generic fibre is exactly the given Drinfeld module over $F$.
\end{enumerate}
\end{rem}

Let $\phi$ be a Drinfeld $A$-module of rank $r$ over $F$ which is a field of finite degree over $k$, and let $R$ be the integral closure of $A$ in $F$. For any $a\in A$, we write: 
$$\phi_a=a\tau^0+\cdots +\Delta_a\tau^{r\deg(a)}.$$
Let $\mathscr{M}=(\G_{a,\Lcal},\varphi)$ be the normalized minimal model of $\phi$ over $R$. Then $\Delta_a\in \Lcal^{1-q^{r\deg(a)}}$. 


\begin{rem}
We can identify the invertible sheaves over $\spec(R)$ with the fractional ideals of $R$ in $F$ by isomorphisms of $R$-modules. Without causing any ambiguity, we implicitly use this identification in this paper and think of the invertible sheaves over $\spec(R)$ as fractional ideals of $R$ in $F$. Thus, we do not distinguish invertible sheaves over $\spec(R)$ and fractional ideals of $R$ in $F$. Moreover, we do not distinguish prime ideals of $R$ and the corresponding valuations on $F$. That is, if we take a prime ideal $v\in \spec(R)$, by $v(I)$ we mean the valuation of the fractional ideal $I$ at $v$. 
\end{rem}

\begin{prop}\label{gmp}
Let $\phi,R,F$ be as above.
\begin{enumerate}
\item If $\phi$ has everywhere good reduction on $R$, then the minimal model $\mathscr{M}$ of $\phi$ over $R$ is a Drinfeld $A$-module over $R$.
\item If $\phi$ has everywhere good reduction on $R$, $\phi'$ is another Drinfeld module over $F$ and $f:\phi\ra \phi'$ is an isogeny of Drinfeld modules over $F$, then $f$ is an isogeny of their normalized minimal models.
\end{enumerate}
\end{prop}

\begin{proof}
To prove the first statement, we can assume that $\mathscr{M}$ is normalized so that we are left to show that $\phi_a$ is finite for each $0\neq a\in A$. As $\phi$ has everywhere good reduction on $R$, we have $v(\Delta_a)=v(\Lcal^{1-q^{r\deg(a)}})$ for every place $v$ on $R$  \cite[Example, p.301]{tag1993}. This means that $\Delta_a$ corresponds to an isomorphism in $\End(\G_{a,\Lcal})$. By \cite[Proposition 1.2.6]{leh09} we see $\phi_a$ is finite for any $a\in A$. So the first statement is true.\\

To prove the second statement, we first write:
$$\phi'_a=a\tau^0+\cdots+\Delta_a'\tau^{r\deg(a)},\ \forall a\in A;\quad f=f_0\tau^0+\cdots +f_n\tau^n.$$For any $a\in A$ we have $f\phi_a=\phi_a'f$. By comparing the coefficients we get
\begin{equation}
f_n\Delta_a^{q^n}=\Delta_a'f_n^{q^{r\deg(a)}}.
\end{equation}
Let $\mathscr{M}'=(\G_{a,\Lcal'},\varphi')$ be the normalized minimal model of $\phi'$. Then $f$ extends uniquely to a morphism from $\mathscr{M}$ to $\mathscr{M}'$ whose generic fibre is $f$ \cite[Proposition 2.5]{tag1993}. We therefore denote the two morphisms by $f$ interchangeably without causing any ambiguity. Since $\phi$ has good reduction everywhere, so does $\phi'$. The above argument again implies $v(\Delta_a')=v(\Lcal'^{1-r\deg(a)})$. From (1) we see for every place $v$ on $R$
$$v(f_n)=v(\Lcal')-q^nv(\Lcal)=v(\Lcal'\Lcal^{-q^n}).$$Thus by the same argument in the proof of the first statement, $f$ is finite, hence an isogeny.
\end{proof}

\begin{rem}
This proposition actually indicates an analogue of a well-known result that if an abelian variety over a number field has good reduction everywhere then its Néron model is an abelian scheme, and moreover, an isogeny between abelian varieties with everywhere good reductions extends to a finite flat homomorphism between their Néron models. In our case, the flatness of $f$ is in consequence of the finiteness since a homomorphism between two line bundles is quasi-finite if and only if it is flat \cite[Proposition 1.2.5]{leh09}.
\end{rem}

\section{Heights}

In this section, we study Taguchi heights, Weil heights and graded heights. In particular, we calculate the variation of Taguchi heights and graded heights of rank 2 Drinfeld $\fq[t]$-modules under an isogeny.\\

\noindent {\large \textbf{Taguchi heights}}\medskip 

Let $A,k,k_{\infty}$ and $\Cinf$ be as in Section 2. In \cite{tag1993}, Taguchi introduced his so-called differential heights of Drinfeld $A$-modules which are now called Taguchi heights. He defines this concept in the case when the lattices associated to Drinfeld $A$-modules are free. Wei generalizes Taguchi's definition for arbitrary Drinfeld $A$-modules \cite[Section 5]{wei20}. We copy the following definition from Wei. For more details, the reader could refer to \cite[Section 4]{wei17} or \cite[Remark 2.10]{wei20}.

\begin{defn}Let $\Lambda$ be an $A$-lattice of rank $r$ in $\Cinf$, and let $\Ocal_{\infty}$ be the ring of $\infty$-adic integers in $k_{\infty}$. Choose an orthogonal $k_{\infty}$-basis $\{\lambda_i\}_{i=1}^r$ of $k_{\infty}\otimes \Lambda$ such that:
\begin{enumerate}[(1)]
\item $\lambda_i\in \Lambda$ for $1\leq i\leq r$;
\item $|a_1\lambda_1+\cdots+a_r\lambda_r|_{\infty}=\max\{|a_i\lambda_i|_{\infty}: 1\leq i\leq r\}$ for all $a_1,...,a_r\in k_{\infty}$;
\item $k_{\infty}\otimes \Lambda=\Lambda+(\Ocal_{\infty}\lambda_1+\cdots +\Ocal_{\infty}\lambda_r)$.
\end{enumerate}
The \textit{covolume} $D_A(\Lambda)$ of the $A$-lattice $\Lambda$ is defined as follows:
$$D_A(\Lambda):=q^{1-g_k}\cdot \left(\frac{\prod_{i=1}^r|\lambda_i|_{\infty}}{\#\left(\Lambda\cap \left(\Ocal_{\infty}\lambda_1+\cdots +\Ocal_{\infty}\lambda_r\right)\right)}\right)^{\frac{1}{r}}=\left(\frac{\prod_{i=1}^r|\lambda_i|_{\infty}}{\#(\Lambda/(A\lambda_1+\cdots+A\lambda_r))}\right)^{\frac{1}{r}},$$where $g_k$ is the genus of the field $k$.
\end{defn}

Let $F/k$ be a finite field extension and $R$ be the integral closure of $A$ in $F$. For each infinite place $w$ of $F$ lying over $\infty$, we can embed $F$ into $\C_{\infty}$ via $w$. Let $F_w$ be the completion of $F$ with respect to $w$ and we embd $F_w \xhookrightarrow{}  \Cinf$. A \textit{metrized line bundle} $(\Lcal, \|\cdot \|)$ on $\spec(R)$ is a projective $R$-module $\Lcal$ of rank $1$, together with norms $$\|\cdot\|_w:\Lcal \otimes_RF_w\ra \R$$for all infinite places $w$ of $F$. The \textit{degree} $\deg(\Lcal,\|\cdot \|)$ of a metrized line bundle $(\Lcal,\|\cdot \|)$ on $R$ is $$\deg(\Lcal, \|\cdot\|):=\log\#(\Lcal/lR)-\sum_{w|\infty}\epsilon_w\log \|l\|_w$$for some $l\in \Lcal$, and $\epsilon_w$ is the local degree at $w$. It is independent of the choice of $l$ by the product formula. We note that by taking a norm we implicitly indicate an extension of the absolute value $|\cdot |_{\infty}$ on $k$. In this above equation, the extension of absolute value is taken as the one remaining unchanged on $k$ while our normalization below is different. The reader should note in this definition we take the infinite absolute value on $F_w$ such that it remains the same on $k$.\\

Let $\phi$ be a Drinfeld $A$-module of rank $r$ over $F$ and $\mathscr{M}=(\G_{a,\Lcal},\varphi,f)$ the minimal model of $\phi$ over $R$, where $\textbf{E}=(\G_{a,\Lcal},\varphi)$ is an $A$-module scheme and $\Lcal$ is an invertible sheaf over $\spec(R)$. Since $\G_{a,\Lcal}\ra \spec(R)$ is smooth of finite type and relative dimension $1$, we see $\Omega_{\G_{a,\Lcal}/R}^1$ is locally free of rank $1$. Let $e:\spec(R)\ra \G_{a,\Lcal}$ be the unit section, then we set $$\omega_{\textbf{E}/R}:=e^*(\Omega_{\G_{a,\Lcal}/R}^1).$$Thus $\omega_{\textbf{E}/R}\cong \Lcal^{-1}$ which is the inverse of $\Lcal$ in $\pic(R)$. Without causing any ambiguity, we treat $\omega_{\textbf{E}/R}$ as a rank $1$ projective $R$-module.  Let $w$ be an infinite place of $F$ and $\textbf{E}_w$ be the Drinfeld module over $F_w$ by extension of scalars $R\ra F_w$. Let $\Lambda_w$ be the corresponding $A$-lattice of rank $r$ in $\Cinf$ (\cite[Theorem 4.6.9]{go98} or \cite[Proposition 2.1.5]{leh09}). If $x$ is the coordinate function of $\G_{a}/F_w$, then $dx$ is a generator of $\omega_{\textbf{E}_w/F_w} (\cong \Lcal^{-1}\otimes_RF_w)$. We put a metric $\|\cdot\|_w$ on $\omega_{\textbf{E}_w/F_w}$ by $$\|dx\|_w:=D_A(\Lambda_w).$$
\begin{defn}
Let $\phi$ be a Drinfeld $A$-module of rank $r$ over $F$ and $\mathscr{M}=(\G_{a,\Lcal},\varphi,f)$ be its minimal model over $R$, where $\textbf{E}=(\G_{a,\Lcal},\varphi)$ is an $A$-module scheme over $R$. The \textit{Taguchi height} of $\phi$ over $F$ is $$h_\ta(\phi/F):=\frac{1}{[F:k]}\deg(\omega_{\textbf{E}/R},\|\cdot\|),$$where the metric $\|\cdot\|$ is given as above.
\end{defn}

It is obvious the Taguchi height of a Drinfeld $A$-module $\phi$ depends on the choice of the field $F$. However, it will remain unchanged when taking a finite field extension of $F$ if $\phi$ has everywhere stable reduction over $F$. Since every Drinfeld $A$-modules has everywhere potential stable reduction, we can define the \textit{stable Taguchi height} of $\phi$ to be $$h_{\ta}^{\text{st}}(\phi):=h_{\ta}(\phi/F'),$$
where $F'$ is a field of finite degree over $F$ on which $\phi$ has everywhere stable reduction. The following isogeny lemma is standard.

\begin{lem}\label{isogl}
Let $f:\phi_1 \ra \phi_2$ be an isogeny of Drinfeld $A$-modules over $F$ with everywhere stable reduction. Then we have:
$$\sth(\phi_2)-\sth(\phi_1)=\frac{1}{r}\log|\deg(f)|-\frac{1}{[F:k]}\log\#(R/D_f),$$where $D_f$ is the different of $f$ (cf. \cite[Section 5.4]{tag1993} or \cite[Section 1.3]{ray1985}).
\end{lem}

\begin{rem}\label{adr}
Let $G$ be the kernel of the induced homomorphism of the minimal models of $\phi_1$ and $\phi_2$. If $f:\phi_1 \ra \phi_2$ is an isogeny of Drinfeld $A$-modules over $F$ with everywhere good reduction, then by Proposition \ref{gmp} we see $f$ induces an isogeny between the minimal models. In this case, according to \cite[Equation 4.9.6]{ill1985}, $D_f$ is the absolute different of $G$ \cite[Apendice, D\'efinition 8]{ray1974}.
\end{rem}

If $\phi$ is a Drinfeld $A$-module of rank $r$ over $F$, we set:
$$v(\phi):=-\min_{a\in A\backslash\{0\}}\min_i\left\{\frac{v(a_i)}{q^i-1}:1\leq i\leq r\deg(a)\right\},$$
where the $a_i$'s are coefficients in $\phi_a=a\tau^0+\sum_{i=1}^{r\deg(a)}a_i\tau^i.$

\begin{lem}\label{dif}
	Let $f:\phi_1\ra \phi_2$ be an isogeny of Drinfeld $A$-modules over $F$ with everywhere good reduction on $R$. If $v\in \spec(R)$ is a finite place, then 
	$$v(D_f)=v(f_0)+v(\phi_1)-v(\phi_2),$$where $f_0=\partial (f)$ is the coefficient of the linear term.
\end{lem}

\begin{proof}
Let $\mathscr{M}_1$ and $\mathscr{M}_2$ be the normalized minimal models over $R$ of $\phi_1$ and $\phi_2$ respectively. It will suffice to prove for the local cases, i.e. we may assume that $R$ is a discrete valuation ring. Suppose $v$ is the valuation on $R$. Let us first look at the case when the normalized minimal models $\mathscr{M}_1=(\spec(R[X]),\varphi_1,\text{Id})$ and $\mathscr{M}_2=(\spec (R[Y]),\varphi_2, \text{Id})$. We use $f$ to denote the isogeny between $\spec (R[X])$ and $\spec (R[Y])$. We denote $f^{\#}:R[Y]\ra R[X]$ the corresponding homomorphism of rings. Thus we have
$$f^{\#}(Y)=f_0X+f_1X^q+\cdots +f_nX^{q^n} \in R[X].$$
The kernel $G$ of $f$ is then given by $\spec (R[X]/(f^{\#}(Y)))$. By \cite[Equations 4.9.5, 4.9.6]{ill1985}, the absolute different of $G$ is $(f_0)\subset R$. Thus $D_f=(f_0)$ (Remark \ref{adr}). By the construction of minimal models \cite[Proposition 2.2]{tag1993}, we note that in this case $v(\phi_1)=v(\phi_2)=0$. Hence our claim is true in this case.\\

To prove the general cases, let $\G_{a,\Lcal_i}$ be the line bundle of the normalized minimal model of $\phi_i$ for $i=1,2$. We assume $\Lcal_i=(a_i)$ to be a fractional ideal for some $a_i\in F$. Thus we have 
$$f_j\in (a_2a_1^{-q^j}).$$
In particular, $f_0=a_2a_1^{-1}b$ for some $b\in R$. Now apply the same argument above with the variables $X$ replaced by $a_1^{-1}X$ and $Y$ replaced by $a_2^{-1}Y$, we see $D_f=(b)=(f_0a_1a_2^{-1})$. Therefore we have $$v(D_f)=v(f_0)+v(a_1)-v(a_2).$$This proves our claim.
\end{proof}

\noindent {\large \textbf{Logarithmic heights}} \medskip

The absolute value $|\cdot|_{\infty}$ on $k$ naturally extends to a unique absolute value on $\C_{\infty}$, which restricts to the same values on $k_{\infty}$ and we denote it by $|\cdot|$. If $F/k$ is a finite field extension and $w$ is a place of $F$ lying over $v\in M_k$, we normalize the absolute value associated to $w$ as 
$$|y|_w=|\text{N}_{F_w/k_v}(y)|_v^{\frac{1}{[F:k]}},\ \forall y\in F.$$
Since $k$ has degree of imperfection 1 (see the Remark \ref{impf} below), for any place $v\in M_k$ we have: 
\begin{equation}\label{pdf}
F\otimes_kk_v\cong \prod_{w|v}F_w.
\end{equation}
Let $M_F$ be the set of places $w$ normalized as above. By $(\ref{pdf})$ we have following two properties:
\begin{itemize}
\item Product formula: For every $y\in F$, $\sum_{w\in M_F}\log|y|_w=0.$
\item Extension formula: $[F:k]=\sum_{w|v}[F_w:k_v].$
\end{itemize}

\begin{rem}\label{impf}
If $F$ is a field of characteristic $p\neq 0$, by degree of imperfection of $F$ we mean the number $n$ such that $[F:F^p]=p^n$. The isomorphism in $(2)$ holds in a more general case of simple extension. A theorem from Becker and MacLane \cite[Theorem 6]{bm1940} tells us any finite extension $L/F$ can be generated by at most $\max\{1,n\}$ elements.
\end{rem}

Let $\overline{k}$ be the algebraic closure of $k$ in $\Cinf$ and we denote by $\Pro^n(\overline{k})$ the $n$-dimensional projective space over $\kcl$. If $\textbf{x}=(x_0:\cdots:x_n)\in \Pro^n(\kcl)$ and $F$ is a  finite extension of $k$ containing these coordinates, then the \textit{Weil height} of $\textbf{x}$ is: $$h(\textbf{x}):=\sum_{w\in M_F}\max_j \log |x_j|_w.$$As in the number field case, this definition is independent of the choice of both the field $F$ and the coordinates. The definition of Weil heights of points in affine space is naturally obtained by embedding the affine space to a projective space. In particular, for $x\in \kcl$ and $F$ a finite extension of $k$ containing $x$ we have 
$$h(x)=\sum_{w\in M_F}\log \max\{1,|x|_w\}=\sum_{w\in M_F}\log^+|x|_w.$$
The proof of following lemma is almost the same as that of \cite[Proposition 1.6.6]{bg2006}. We give a proof here for the convenience of the reader.

\begin{lem}\label{htl}
Suppose $A=\fq[t]$. Let $\alpha \in k^{\sep}$ of degree $d$ and $f(X)$ be the minimal polynomial of $\alpha$ over $A$ with leading coefficient $a_d$ and roots $\alpha_j$, $j=1,...,d$. Then $$dh(\alpha)=\log|a_d|+\sum_{j=1}^d\log^+|\alpha_j|.$$
\end{lem}

\begin{proof}
	Let $F/k$ be a finite Galois extension that contains $\alpha$. Let $G:=\gal(F/k)$. The set $\{\sigma(\alpha)\}_{\sigma\in G}$ contains every conjugate of $\alpha$ exactly $[F:k]/d$ times. We have 
	$$a_d\prod_{\sigma\in G}(X-\sigma(\alpha))^{\frac{d}{[F:k]}}=f(X).$$
	Apply Gauss's lemma \cite[Lemma 1.6.3]{bg2006}, we get for any finite place $w\in M_F$
	\begin{equation}\label{gl}
		|a_d|_w\prod_{\sigma\in G}\max\{1,|\sigma(\alpha)|_w\}^{\frac{d}{[F:k]}}=1.
	\end{equation}
	Thus, we have 
	\begin{align*}
		[F:k]h(\alpha) & = \sum_{w\in M_F}\sum_{\sigma\in G}\log^+|\sigma(\alpha)|_w,\\
		& = \sum_{w|\infty}\sum_{\sigma\in G}\log^+|\sigma(\alpha)|_w+\sum_{w\nmid \infty}\sum_{\sigma\in G}\log^+|\sigma(\alpha)|_w\\
		& = \sum_{\sigma\in G}\sum_{w|\infty}\log^+|\sigma(\alpha)|_w+\sum_{w\nmid \infty}\sum_{\sigma\in G}\log^+|\sigma(\alpha)|_w\\
		& = \frac{[F:k]}{d}\sum_{w| \infty}\sum_{j=1}^d\log^+|\alpha_j|_w-\frac{[F:k]}{d}\sum_{w\nmid \infty}\log |a_d|_w,\ \text{ by (\ref{gl})}\\
		& = \frac{[F:k]}{d}\sum_{w|\infty}\sum_{j=1}^d\log^+|\alpha_j|_w+\frac{[F:k]}{d}\sum_{w|\infty}\log |a_d|_w,\ \text{\ by product formula}\\
		& = \frac{[F:k]}{d}\left( \log|a_d|+\sum_{w|\infty} \sum_{j=1}^d \frac{[F_w:k_{\infty}]}{[F:k]}\log^+|\sigma_w(\alpha_j)|\right),\\
		& = \frac{[F:k]}{d}\left(\log|a_d|+\sum_{w|\infty}\frac{[F_w:k_{\infty}]}{[F:k]}\sum_{j=1}^d\log^+|\alpha_j|\right)\\
		& = \frac{[F:k]}{d}\left(\log |a_d|+\sum_{j=1}^n\log^+|\alpha_j|\right).
	\end{align*}
	Now the lemma follows.
\end{proof}

\begin{defn}
	Let $\phi$ be a Drinfeld $A$-module of rank $r$ over $F$. The \textit{global graded degree} $h_G(\phi)$ (resp. \textit{local graded degree $h_G^w(\phi/F)$ at a place $w$ over $F$}) of $\phi$ is 
	$$h_G(\phi):=\frac{1}{[F:k]}\sum_{w\in M_F}\deg(w)w(\phi)\ \left(\text{resp}.\ h_G^w(\phi/F):=\frac{\deg(w)w(\phi)}{[F:k]}\right).$$
\end{defn}

\begin{rem}\label{rkgh}
\begin{enumerate}[(1)]
	\item It is obvious the global graded height of $\phi$ does not depend on the choice of the field $F$ and it is invariant under isomorphisms, while a local graded height of $\phi$ will not satisfy such properties.
	\item The global graded height is a direct interpretation of ``finite'' Taguchi height (cf. \cite[Definition 2.3]{tag1991}).
	\item Let $\phi$ be a Drinfeld $\fq [t]$-module of rank $r$ over $F$. Then it is characterised by:
	$$\phi_t=t\tau^0+g_1\tau+\cdots +g_r\tau^r,\ g_i\in F, g_r\neq 0.$$
	Let $m=\text{lcm}\{q-1,...,q^r-1\}$. We set $J:=(j_1:\cdots :j_r)\in \Pro^{r-1}(\kcl)$ where 
	$$j_i=g_i^{m/(q^i-1)},\ \text{for}\ i=1,...,r.$$
	If $r=2$, then $j_{\phi}:=j_1/j_2$ is the $j$-invariant of the Drinfeld $\fq[t]$-module $\phi$. It plays the same role as the $j$-invariant of elliptic curves. The global graded height of $\phi$ is then given by:
	$$h_G(\phi)=\sum_{w\in M_F}\max_{1\leq i\leq r}\log|g_i|_w^{1/(q^i-1)}.$$
	Thus the global graded height coincides the one in \cite[Equation 6]{bpr21}. It is obvious that $mh_G(\phi)=h(J)$.
\end{enumerate}
\end{rem}

\begin{prop}\label{ghga}
Let $\phi$ be a Drinfeld $A$-module of rank $r$ over $F$ such that $F/k$ is a separable extension. For any $\sigma\in \gal(k^{\sep}/k)$, we denote by $\sigma(\phi)$ the Drinfeld $A$-module obtained by acting $\sigma$ on the coefficients of a Drinfeld $A$-module $\phi$. Then we have $$h_G(\phi)=h_G(\sigma(\phi)).$$
\end{prop}

\begin{proof}
By Remark \ref{rkgh} (1) we may assume $F/k$ is a Galois extension so that for any $\sigma\in \gal(F/k)$ the Drinfeld $A$-module $\sigma(\phi)$ is defined over $F$. For any places $v\in M_k$ and $w\in M_F$ such that $w$ lies over $v$, we see $w\circ \sigma$ is again a place lying over $v$. Thus $\sigma$ permutes the places lying over $v$. By a result of algebraic number theory, $w$ and $w\circ \sigma$ have the same degree. Therefore
$$h_G(\sigma(\phi))=\sum_{v\in M_k}\sum\limits_{\substack{w\in M_F\\ w|v}} \deg(w\circ \sigma)w\circ\sigma(\phi)=\sum_{v\in M_k}\sum\limits_{\substack{w\in M_F\\ w|v}} \deg(w)w(\phi)=h_G(\phi).$$
\end{proof}

\begin{thm}\label{vth}
	Let $f:\phi_1\ra \phi_2$ be an isogeny of Drinfeld $A$-modules over $F$ with everywhere good reduction on $\spec(R)$. Then we have:
	$$\sth(\phi_2)-\sth(\phi_1)=\frac{1}{r}\log|\deg(f)|-\sum\limits_{\substack{w|\infty\\ w\in M_F}}\log|f_0|_w+h_G^{\f}(\phi_2)-h_G^{\f}(\phi_1),$$ where $f_0$ is the linear coefficient of $f$ and for $i=1,2$, $h_G^{\f}(\phi_i)$ is the sum of the local graded heights running over all finite places of $F$.
\end{thm}

\begin{proof}
By applying Lemma \ref{isogl} and Lemma \ref{dif} we obtain:
\begin{equation}\label{ist}
\sth(\phi_2)-\sth(\phi_1) = \frac{1}{r}\log|\deg(f)|- \frac{1}{[F:k]}\sum\limits_{\substack{w\in M_F \\ w \nmid \infty}} \deg(w)(w(f_0)+w(\phi_1)-w(\phi_2)).
\end{equation}
By applying the product formula, we get:
\begin{equation}\label{pr}
\sum\limits_{\substack{w\in M_F \\ w \nmid \infty}} \deg(w)w(f_0) = - \sum\limits_{\substack{w\in M_F \\ w | \infty}} \deg(w)w(f_0).
\end{equation}
Under our normalization, we have 
\begin{equation}\label{nomab}
\log|f_0|_w=\frac{-\deg(w)w(f_0)}{[F:k]}.
\end{equation}
Now substitute (\ref{nomab}) and (\ref{pr}) to (\ref{ist}) we obtain our formula.
\end{proof}

\noindent {\large \textbf{Heights of Drinfeld $A$-modules with complex multiplication}} \medskip

We first recall the CM theory for Drinfeld modules. Let $\phi$ be a Drinfeld $A$-module of rank $r$ over $\Cinf$. We say $\phi$ has \textit{complex multiplication} if the ring of endomorphisms $\Ocal:=\End(\phi)$ is a projective $A$-module of rank $r$, and $K:=\Ocal\otimes_A k$ is called the \textit{CM field of $\phi$}. In this case, $K/k$ is an imaginary extension of degree $r$. Here by an \textit{imaginary extension} we mean there is only one place of $K$ extending the infinite place $\infty$ of $k$. As with the case of abelian varieties, there is also a standard theory of complex multiplication for Drinfeld modules:

\begin{thm}(Main Theorem of Complex Multiplication) Let $\phi$ be a Drinfeld $A$-module of rank $r$ over $\Cinf$ with complex multiplication. Let $\Ocal$ be the ring of endomorphisms of $\phi$ and $K$ be the CM field. The following statements are true:
	\begin{enumerate}
\item There is a finite extension $H_{\Ocal}/K$ such that $\gal(H_{\Ocal}/K)\cong \pic(\Ocal)$ via the Artin map. The field $H_{\Ocal}$ is the ring class field of $\Ocal$. The prime $\infty$ of $K$ splits completely in $H_{\Ocal}$, and $\Ocal$ is unramified outside $\mathcal{C}$ which is the conductor of $\Ocal$, i.e. the largest common ideal of $\Ocal$ and $\Ocal_K$.
\item $\phi$ has good reduction at every finite place of $H_{\Ocal}$.
\item If $r=2$ and $A=\fq[t]$, then the $j$-invariant $j_{\phi}$ of $\phi$ is integral over $A$ and $H_{\Ocal}=K(j_{\phi})$.
	\end{enumerate}
\end{thm}

The reader could find a proof to the above statements in \cite{dh1979}, as well as a complete treatment of theory of complex multiplications for Drinfeld $A$-modules.\\

Suppose $\phi$ is a Drinfeld $A$-module of rank $r$ with CM by an order $\Ocal$ in a CM field $K$. Recall that a \textit{proper ideal} $I$ of $\Ocal$ is a fractional ideal of $\Ocal$ in $K$ such that
$$\{x\in K: x\cdot I\subset I\}=\Ocal.$$
 We denote by $\pr(\Ocal)$ the monoid of proper fractional ideals of $\Ocal$ quotient by principal ideals. It is then obvious $\pic(\Ocal)\subset \pr(\Ocal)$ and $\pic(\Ocal)$ has a natural action on $\pr(\Ocal)$. Since $\phi$ has CM by $\Ocal$, its associated lattice is isomorphic to a proper ideal $I_{\phi}$ of $\Ocal$. 

\begin{lem}\label{ivgh}
Let $\phi_1$ and $\phi_2$ be two Drinfeld $A$-modules of rank $r$ with CM by the same order $\Ocal$, and $I_1$ (resp. $I_2$) be a proper ideal of $\Ocal$ such that the associated lattice of $\phi_1$ (resp. $\phi_2$) is isomorphic to $I_1$ (resp. $I_2$). If $I_1$ and $I_2$ are in the same orbit of $\pr(\Ocal)$ under the action of $\pic(\Ocal)$, then $h_G(\phi_1)=h_G(\phi_2)$. In particular, if $\pr(\Ocal)=\pic(\Ocal)$ then all the Drinfeld $A$-modules with CM by $\Ocal$ have the same graded height.
\end{lem}

\begin{proof}
Without loss of generality, we assume $\phi_i$ has associated lattice $I_i$ where $i=1,2$. We choose an invertible ideal $J\in \pic(\Ocal)$ such that $I_1=J^{-1}\cdot I_2$. Thus $I_1$ is homothetic to the lattice associated to $J*\phi_2$ (cf. \cite[Proposition 5.10 and Equation (5.18)]{dh1979}). On the other hand, $J*\phi_2$ is isomorphic to a Drinfeld $A$-module $\phi_2'$ obtained by a Galois action on the coefficients of $\phi_2$ \cite[Theorem A.1 (2)]{wei20}. We note that we can always choose suitable $I_1$ and $I_2$ to make $J$ integral so that our argument makes sense. Now by Remark \ref{rkgh} (1) and Proposition \ref{ghga} we complete our proof.
\end{proof}

For the rest of this paper, we always fix $A=\fq[t]$. A proof of the following result for elliptic curves is due to Nakkajima and Taguchi \cite{nt1991}. We only make a few arguments here to adapt their formula to Drinfeld $A$-modules.

\begin{prop}\label{abdf}
Let $\phi_1$ and $\phi_2$ be two Drinfeld $A$-modules of rank 2 with CM by $\Ocal_K$ and $\Ocal$ respectively, where $K$ is an imaginary quadratic field and $\Ocal_K$ (resp. $\Ocal$) is a maximal (resp. arbitrary) order. We write $\Ocal=A+f_0\Ocal_K$ for some $f_0\in A$. If $F/K$ is a finite field extension such that both $\phi_1$ and $\phi_2$ are defined over $F$ with everywhere good reduction then
$$\sth(\phi_2)-\sth(\phi_1)=\frac{1}{2}\log |f_0|-\frac{1}{2}\sum_{v|f_0}\deg(v)e_{f_0}(v),$$
where $v$ runs over all monic prime factors of $f_0$ and for $l:=q^{\deg(v)}$
$$e_{f_0}(v)=\frac{(1-\chi(v))(1-l^{-v(f_0)})}{(l-\chi(v))(1-l^{-1})},$$
and $\chi(v)=1$ if $v$ splits in $K$; $\chi(v)=0$ if $v$ ramifies in $K$; $\chi(v)=-1$ if $v$ is inert in $K$.
\end{prop}

\begin{proof}
First we note that the Taguchi height of rank 2 Drinfeld $A$-modules with CM does not depend on the choice of lattice that analytically generates the corresponding Drinfeld $A$-module. So we may assume that $\phi_1$ is given by $\Ocal_K$ and $\phi_2$ is given by $\Ocal$, and an isogeny $f:\phi_1\ra \phi_2$ given by $f_0\Ocal_K\subset \Ocal$.\\

For the case when $\chi(v)=1$, from Theorem \ref{cord} we see $\phi_1$ has ordinary reduction at any place lying over $v$. Applying the same argument from \cite[Proposition 4]{nt1991} we get $e_{f_0}(v)=0$. The reduction process in the argument for Drinfeld $A$-modules is given by Theorem \ref{rdp}. For the case of supersingular reduction, the argument for our case is exactly the same as \cite[Section 2.2]{nt1991} with only one modification that we take $l=q^{\deg(v)}$.
\end{proof}

\begin{cor}\label{vghr2}
Assume  the same conditions as in Proposition \ref{abdf}. The following formula is true
$$h_G(\phi_2)-h_G(\phi_1)=\log|f_0|-\frac{1}{2}\sum_{v|f_0}\deg(v)e_{f_0}(v)+h_G^{\infty}(\phi_2')-h_G^{\infty}(\phi_1'),$$
where $\phi_1'$ is the Drinfeld $A$-module given by the lattice $\Ocal_K$ and $\phi_2'$ is given by $\Ocal$, and $h_G^{\infty}(\phi_i')$ is the sum of local graded heights of $\phi_i'$ at infinite places for $i=1,2$.
\end{cor}

\begin{proof}
By Lemma \ref{ivgh} we can choose $\phi_i$ to be $\phi_i'$, $i=1,2.$ It is then a trivial consequence of Theorem \ref{vth} and Proposition \ref{abdf}.
\end{proof}

\section{Arithmetic on quadratic fundamental domain}

We assume our Drinfeld $A$-module $\phi$ has rank 2 with CM for the rest this paper. Let $\Omega:=\Cinf\backslash k_{\infty}$ be the Drinfeld upper-half plane so that $\text{PGL}_2(A)\backslash \Omega$ are the $\Cinf$-points of the coarse moduli space of rank 2 Drinfeld $A$-modules over $\Cinf$. As in the case of elliptic curves, there are bijections: $$\text{PGL}_2(A)\backslash \Omega \xrightleftharpoons{\quad \quad} \{\text{Lattices of rank 2 in }\Cinf /\cong \}\xrightleftharpoons{\quad\ \quad}\Cinf.$$
Thus we obtain a natural $j$-function $j:\Omega \rightarrow \Cinf$. The set of $j$-invariants of rank 2 Drinfeld $A$-modules with CM is precisely the image of the $j$-function at imaginary quadratic arguments.\\

Unfortunately, the Drinfeld upper-half plane doesn't have a good geometry as the Poincar\'e upper-half plane of complex numbers does. However, we can still define the quadratic fundamental domain.

\begin{defn}(cf. \cite[Definition 3.4]{fb05})
The \textit{quadratic fundamental domain} is
\begin{equation}
\begin{aligned}\label{qfd}
\mathcal{D}=\{z\in \Omega &: z\text{ satisfies an equation of the form }az^2+bz+c=0, \\
& \quad \text{where } a,b,c\in A,\ a\text{ is monic, }|b|<|a|\leq |c|\text{ and}\\ & \quad \text{gcd}(a,b,c)=1\}.
\end{aligned}
\end{equation}
\end{defn}

Any rank 2 lattice corresponding to a CM Drinfeld module is homothetic to $\Lambda_z$ for some $z\in \mathcal{D}_K$, where $\Lambda_z$ denotes the lattice generated by $z$ and 1, and $\mathcal{D}_K:=\mathcal{D}\cap K$ for some imaginary quadratic extension $K/k$ in $\Cinf$. Unlike the case of elliptic curves, such $z$ is not necessarily unique.

\begin{prop}\label{extp}\cite[Proposition 1.2.1]{breuer02}
Let $q$ be odd and $K$ be a quadratic extension of $k$. Then $K$ is a Kummer extension and can be written in the form $K=k(\sqrt{\delta})$ for some square-free $\delta \in A$. Let $m=\deg(\delta)$. Then we have
\begin{enumerate}
\item The place $\infty$ ramifies in $K$ if and only if $m$ is odd;
\item The place $\infty$ is inert in $K$ if and only if $m$ is even and the leading coefficient of $\delta$ is not a square in $\fq$;
\item The place $\infty$ splits in $K$ if and only if $m$ is even and the leading coefficient of $\delta$ is a square in $\fq$.
\end{enumerate}
\end{prop}

Let $\delta\in A$ be a polynomial of odd degree or, even degree with the leading coefficient not being a square in $\fq$, and let $\sqrt{\delta}\in \overline{k}$ be a root of $X^2-\delta$. The field $K:=k(\sqrt{\delta})$ is an imaginary quadratic field by Proposition \ref{extp}. Let $\Ocal_K$ be its maximal order. If $\Ocal\subset \Ocal_K$ is a suborder of discriminant $\delta$, then there exsits some $f\in A$ such that $\Ocal=A+f\Ocal_K$ and such $f$ is called the \textit{conductor} of $\Ocal$. The discriminant $\delta_0$ of $\Ocal_K$ is called the \textit{fundamental discriminant} of $K$ and $K=k(\sqrt{\delta_0})$. The discriminant of $\Ocal$ is $\delta=4f^2\delta_0$. \\

We denote by $T_{\delta}$ the set of triples $(a,b,c)$ with $a,b,c\in A$ such that $b^2-4ac=\delta$ and satisfying (\ref{qfd}). For $(a,b,c)\in T_{\delta}$ we set $$z(a,b,c)=\frac{-b+\sqrt{\delta}}{2a}\in K=k(\sqrt{\delta}).$$

The map $(a,b,c)\mapsto j(z(a,b,c))$ is a bijection from $T_{\delta}$ to the Galois conjugates of $j(z)$.


\begin{lem}\label{raml}
If $\infty$ ramifies in $K$, then there doesn't exist any $z\in \mathcal{D}_K$ such that $z$ is in the open ball of radius 1 of the point $u$ for any $u\in \F_{q^2}\backslash \fq$.
\end{lem}

\begin{proof}
For some $u\in \F_{q^2}\backslash \fq$, we assume there exists $z\in \mathcal{D}_K$ such that $|z-u|<1$. In this case, we have $|z|=1$ as $|u|=1$. Besides, there exists a triple $(a,b,c)$ satisfying (\ref{qfd}) such that $b^2-4ac=\delta$ and $z$ is root of the equation: $$aX^2+bX+c=0.$$We write $K=k(\sqrt{\delta_0})$, where $\delta_0\in A$ is square-free. By Proposition \ref{extp} and since $K$ is an imaginary quadratic extension of $k$, we see $\delta_0$ is either of odd degree or of even degree with leading coefficients not a square in $\fq$. Also we have $$|az^2+bz|=|c|.$$Since $|bz|=|b|<|a|=|az^2|$, we have $|a|=|az^2+bz|=|c|$, which implies $a$ and $c$ have the same degree and $b$ has degree less than $a$ and $c$. Therefore $\delta$ has even degree. Since $\delta=4f^2\delta_0$, we see $\delta_0$ has even degree, hence with leading coefficients not a square in $\fq$. This is equivalant to saying that $\infty$ is inert in $K$.
\end{proof}

\begin{rem}
If $\delta$ has even degree with leading coefficient not a square in $\fq$, then $\sqrt{\delta}\in \F_{q^2}((\frac{1}{t})).$ Moreover, if there exists some $z\in \mathcal{D}_K$ with discriminant $\delta$ such that $|z-u|<1$, then $\delta$ has leading coefficient $4u^2$. This also suggests such $z$ can be close only to those $u\in \F_{q^2}\backslash \fq$ such that $u^2\in \fq$.
\end{rem}

\begin{prop}\label{ptcp}
Let $u\in \F_{q^2}\backslash \fq$. The number of $(a,b,c)\in T_{\delta}$ with $\deg(\delta)>0$ such that $|z(a,b,c)-u|<\sqrt{|\delta|}^{-1}$ is at most 1.
\end{prop}

\begin{proof}
Since $\delta$ has positive degree, we have $\sqrt{|\delta|}\geq 1$. Thus we have $|z-u|<\sqrt{|\delta|}^{-1}\leq 1$. By the proof of Lemma \ref{raml}, we have$$\delta=\alpha_{2e}t^{2e}+\cdots+\alpha_0,\ \alpha_i\in \fq\  \text{for }i=0,...,2e\ \text{and $e$ is a positive integer.}$$Thus we have $$\sqrt{\delta}=\lambda_et^e+\cdots+\lambda_0+\lambda_{-1}t^{-1}+\lambda_{-2}t^{-2}\cdots$$with coefficients in $\F_{q^2}$. By identifying $(\sqrt{\delta})^2=\delta$ we obtain:
\begin{align*}
\alpha_{2e} & =\lambda_e^2;  \\
\alpha_{2e-1} & = \lambda_e\lambda_{e-1}+\lambda_{e-1}\lambda_e;\\
 & \vdots \\
\alpha_e &= \lambda_e\lambda_0+\lambda_{e-1}\lambda_1+\cdots+\lambda_1\lambda_{e-1}+\lambda_0\lambda_e;\\
& \vdots\\
\alpha_0 &= \lambda_e\lambda_{-e}+\lambda_{e-1}\lambda_{-(e-1)}+\cdots+\lambda_{-(e-1)}\lambda_{e-1}+\lambda_{-e}\lambda_e.
\end{align*}First notice $\lambda_e=2u\neq 0$ because $|z-u|<1$. We first claim that $(2u)^{-1}\lambda_i\in \fq$ when $i=0,1,...,e$.\\

Our claim is trivial when $i=e$. We proceed by induction and suppose it's true for $\lambda_e,...,\lambda_n$ when $0< n\leq e$. From the equations above we have $$\alpha_{e+n-1}=\lambda_e\lambda_{n-1}+\lambda_{e-1}\lambda_{n}+\cdots+\lambda_{n}\lambda_{e-1}+\lambda_{n-1}\lambda_e.$$By multiplying $(2u)^{-2}$ on both sides we obtain $$(2u)^{-2}\alpha_{e+n-1}=(2u)^{-1}\lambda_e(2u)^{-1}\lambda_{n-1}+\cdots+(2u)^{-1}\lambda_{n-1}(2u)^{-1}\lambda_e.$$Since all terms other than $(2u)^{-1}\lambda_e(2u)^{-1}\lambda_{n-1}$ are in $\fq$ and $\lambda_e=2u$, we have $(2u)^{-1}\lambda_{n-1}\in \fq$.\\

Now recall $z=\frac{-b+\sqrt{\delta}}{2a}$ with triple $(a,b,c)\in T_{\delta}$ and $\delta=b^2-4ac$. Then $|z-u|<\sqrt{|\delta|}^{-1}$ is equivalent to $$\deg(a)-\deg(\sqrt{\delta}-b-2au)>\deg(\sqrt{\delta})=\deg(a).$$So we have $\deg(\sqrt{\delta}-b-2au)<0$. Suppose 
$$a=\sum_{i=0}^ea_it^i,\ b=\sum_{i=0}^eb_it^i,\ \text{with all $a_i,b_i\in \fq$}.$$
Thus for all $i=0,...,e$ we have $\lambda_i-2a_iu=b_i$. Since $(2u)^{-1}\lambda_i\in \fq$, we see $b_i=0$ and $a_i=(2u)^{-1}\lambda_i$ for all $i=0,...,e$. Thus, $a,b,c$ are completely determined by $|z(a,b,c)-u|<\sqrt{|\delta|}^{-1}.$
\end{proof}

\section{Bounding $h(J)$}

We prove our main results in this section.\\

\noindent {\large \textbf{Upper bound on $h(J)$}} \medskip

Let $\phi$ be a CM Drinfeld $A$-module of rank 2 over $\Cinf$ and $J$ be its $j$-invariant of degree $d$ over $A$. Let $\Ocal=\text{End}(\phi)$ and $K=\Ocal\otimes_A k$. We denote by $J=J_1,...,J_d$ all the Galois conjugates of $J$, and by $z_1,...,z_d$ the corresponding points in $\mathcal{D}_K$ respectively. Then for each $i$, we have $z_i$ satisfying the equation: $$a_iX^2+b_iX+c_i=0,\ (a_i,b_i,c_i)\in T_{\delta}$$and $b_i^2-4a_ic_i=\delta$ for some $\delta \in A$ that is the discriminant of $\Ocal$. In this subsection we prove:

\begin{prop}\label{ubji}
Assuming the notations above and $J$ is an algebraic unit, we have $$h(J)\leq \left(1+ \frac{q^2-q}{d}\right) (q+1)\log \sqrt{|\delta|}+O_q(1),$$
where $O_q(1)$ is some constant depending only on $q$.
\end{prop}

We first fix some notations. Let $u$ be a point such that $u\in \F_{q^2}\backslash\fq$ and set
 $$|z|_A:=\inf_{a\in A}|z-a|,$$
 $$|z|_{\text{i}}:=\inf_{x\in k_{\infty}}|z-x|.$$
\begin{lem}\label{zml}\cite[Proposition 3.2.5]{breuer02}
If $z\in \mathcal{D}_K,$ then $|z|_{\text{i}}=|z|_A=|z|\geq 1$.
\end{lem}

\begin{lem}\label{tsl}
For each $z_i$, we have $h(z_i)\leq \log\sqrt{|\delta|}.$
\end{lem}

\begin{proof}
	Let $(a,b,c)$ be a triple in $A$ satisfying $(\ref{qfd})$ such that $z_i$ is a root of the equation:
$$aX^2+bX+c=0$$
with discriminant $\delta$. Let $\bar{z_i}$ be the conjugate of $z_i$. Then
$$|z_i\cdot \bar{z_i}|=\frac{|\delta|}{|a^2|}.$$
By Lemma \ref{htl}, we get 
$$2h(z_i)=\log |a|+\log|z_i|+\log|\bar{z_i}|=\log \frac{|\delta|}{|a|}\leq \log |\delta|.$$
This completes the proof.
\end{proof}

\begin{lem}\label{jil}\cite[Lemma 2.6.9]{brown92}
Suppose $z\in \Omega$ such that $|z|_A>q^{-1}$. If $u\in \F_{q^2}\backslash \fq$ and $|z-u|<q^{-1}$, then there exists $\zeta\in \Cinf$ with $|\zeta|<1$ such that 
\begin{align*}
j(z) &= t^qu^{-2}(1-u^{q-1})^{-2}(z-u)^{q+1}(1+\zeta), \\
|j(z)|& = q^q|z-u|^{q+1}.
\end{align*}
\end{lem}


Let $h(J)$ be the Weil height of $J$. Therefore, by Lemma \ref{htl} we have $$dh(J)=\sum_{i=1}^d\log^+|J_i|=\sum_{i=1}^d\log^+|j(z_i)|.$$If we assume $J$ is an algebraic unit, then we have:

\begin{align}\label{wosm}
	dh(J)=dh(J^{-1}) & =\sum_{i=1}^d\log^+|j(z_i)|^{-1} \nonumber \\ 
	& =\sum\limits_{\substack{u\in \F_{q^2}\backslash \fq \\ |z_i-u|<\sqrt{|\delta|}^{-1}}}\log^+|j(z_i)|^{-1}+\sum\limits_{\substack{u\in \F_{q^2}\backslash \fq \\ \sqrt{|\delta|}^{-1}\leq |z_i-u|\leq q^{-1}}}\log^+|j(z_i)|^{-1}+dO_q(1).
\end{align}

The first equality comes from our assumption that $J$ is a unit. To see the third equality holds, we only need to show that 
\begin{equation}\label{bfbl}
	\sum_{q^{-1}\leq |z_i-u|}\log^+|j(z_i)|^{-1}\leq d\lambda_q,
\end{equation}
where $\lambda_q$ is a constant depending only on $q$. We need some arguments to see this. Recall that the fundamental domain for Drinfeld upper half plane (cf. \cite[Theorem 6.4, Proposition 6.6]{gek99}) is given by
$$\mathcal{F}:=\{z\in \Cinf: |z|=|z|_{\text{i}}\geq 1\}.$$
We note that $z$ is a zero of the $j$-function if and only if $z$ is conjugate to some $u\in \F_{q^2}\backslash \fq$ under of M\"obius action of GL$_2(A)$ \cite[(3.9)]{gek99}. By \cite[Corollary 6.7]{gek99}, the zeros of $j$ in $\Fcal$ are exactly these $u\in\F_{q^2}\backslash \fq$. Moreover, $\Fcal$ is closed. This is because if we take a sequence $\{z_n\}\subset \Fcal$ such that $\{z_n\}$ converges to $z\in \Omega$, then for sufficiently large $n$, $|z_n-z|<\epsilon$ and $|z|=|z_n|\geq 1$. Therefore
$$|z|_{\text{i}}=\inf_{x\in k_{\infty}} |z-x|=\inf_{x\in k_{\infty}} |z-z_n+z_n-x|=|z_n|_{\text{i}}=|z_n|=|z|,$$
which implies $z\in \Fcal$. Now the inequality (\ref{bfbl}) can be deduced easily.\\

We also note that under our assumption, necessarily we have $\deg(\delta)>0$.  Indeed, if $\deg(\delta)=0$ then $z_i\in \F_{q^2}\backslash \fq$. In this case $J=0$.

\begin{lem}\label{fte}
Assuming the notations above, we have $$\sum_{|z_i-u|<\sqrt{|\delta|}^{-1}}\log^+|j(z_i)|^{-1}\leq 2(q+1)\log\sqrt{|\delta|}-q.$$
\end{lem}

\begin{proof}
Since $z_i\in \mathcal{D}_K$, we have $|z|_A\geq 1$ by Lemma \ref{zml}. Thus according to Lemma \ref{jil} we have 
\begin{equation}\label{je}
\log |j(z_i)|^{-1}=-q-(q+1)\log|z_i-u|.
\end{equation}
Let $L=k(z_i,u)$, then $[L:k]=2$ if $z_i\in \fq(t)(u)=\F_{q^2}(t)$ and $[L:k]=4$ otherwise. In the first case, there is only one $w\in M_L$ lying over $\infty\in M_k$. Therefore we have $|z_i-u|_w=|z_i-u|$ under our normalization. In the second case, there are two places in $M_L$ lying over $\infty\in M_k$, and we take $w$ to be one of the two. Thus we have $|z_i-u|=|z_i-u|_w^2$. Either way, we find
$$\log|z_i-u|\leq 2\sum_w\log^+|z_i-u|_w=2h(z_i-u).$$
Note the fact that $h(\alpha)=h(1/\alpha)$ for any $\alpha\in \overline{k}$ and we have $$\log|z_i-u|\geq -2h(z_i-u)\geq -2(h(z_i)+h(u))=-2h(z_i).$$Now substitute this inequality to (\ref{je}) and apply Lemma \ref{tsl} we obtain $$\log|j(z_i)|^{-1}\leq 2(q+1)\log\sqrt{|\delta|}-q.$$Since the number of such $z_i$ is at most one by Proposition \ref{ptcp}, we conclude.
\end{proof}

Now we are ready to prove Proposition \ref{ubji}.

\begin{proof}
We are left to estimate 
$$\sum\limits_{\substack{u\in \F_{q^2}\backslash \fq \\ \sqrt{|\delta|}^{-1}\leq |z_i-u|\leq  q^{-1}}}\log^+|j(z_i)|^{-1}.$$
From (\ref{je}) and $\sqrt{|\delta|}^{-1}\leq |z_i-u|$ we have 
$$\sum\limits_{\substack{u\in \F_{q^2}\backslash \fq \\ \sqrt{|\delta|}^{-1}\leq |z_i-u|\leq  q^{-1}}}\log^+|j(z_i)|^{-1}\leq \left(-q+(q+1)\log\sqrt{|\delta|}\right) \cdot \sum\limits_{\substack{u\in \F_{q^2}\backslash \fq \\ \sqrt{|\delta|}^{-1}\leq |z_i-u|\leq  q^{-1}}}1.$$
If the number of $u\in \F_{q^2}\backslash \fq$ such that there exists some $z_i$ such that $|z_i-u|< \sqrt{|\delta|}^{-1}$ is $N$, then we have 
$$\sum\limits_{\substack{u\in \F_{q^2}\backslash \fq \\ \sqrt{|\delta|}^{-1}\leq |z_i-u|\leq  q^{-1}}} \leq d-N.$$
Combining with Lemma \ref{fte}, we get 
\begin{align*}
	dh(J) & \leq 2N(q+1)\log \sqrt{|\delta|}+(d-N)(q+1)\log\sqrt{|\delta|}+dO_q(1)\\
	& =(d+N)(q+1)\log \sqrt{|\delta|}+dO_q(1).
\end{align*}
Thus, we conclude by applying the fact $N\leq q^2-q.$
\end{proof}

\noindent {\large \textbf{Lower bound on $h(J)$}} \medskip

\begin{lem}\label{lbth}
Let $\phi^0$ be a Drinfeld $A$-module of rank 2 with CM by the maximal order $\Ocal_K$ in an imaginary quadratic field $K$. We denote the genus of $K$ by $g_K$. Then we have:
$$\sth(\phi^0)\geq \left(\frac{1}{2}-\frac{1}{\sqrt{q}+1}\right)g_K-\frac{5q-3}{4(q-1)}.$$
\end{lem}

\begin{proof}
By the equation of stable Taguchi height from \cite[Section 5.2.1]{wei20}, we get
$$\sth(\phi^0)=\frac{g_K}{2}\log q_K+\frac{1}{2}\left(\log q_{\infty}-\log q_K\right)-1+\frac{\gamma_K}{2\ln q},$$
where $q_K$ is the cardinality of the field of constants in $K$, $q_{\infty}$ is the cardinality of the residue field of $K_{\infty}$, and $\gamma_K$ is the Euler-Kronecker constant of $K$ \cite[Equation (0.2)]{ih2006}. We remind the reader that the definition of stable Taguchi height in \cite[Equation 5.1]{wei20} is a multiple of our stable Taguchi height by the constant $\ln q$. From \cite[Equation 1.4.6]{ih2006} we get
$$\frac{\gamma_K}{2\ln q}\geq \frac{-g_K}{\sqrt{q}+1}+\frac{q-3}{4(q-1)}.$$
Now using the fact $q_K,q_{\infty}\in \{q,q^2\}$ we get the lower bound.
\end{proof}

\begin{lem}\label{etad}
Let $e_{f_0}(v)$ be as in Proposition \ref{abdf} and $f_0\in A$. Then we have 
$$\frac{1}{2}\sum_{v|f_0}\deg(v)e_{f_0}(v)\leq \frac{9}{4}\log\log |f_0|+C_q,$$
where $v$ runs through all the monic prime factors of $f_0$, and $C_q$ is a computable constant depending on $q$.
\end{lem}

\begin{proof}
First we need a Mertens-type formula for function field, i.e. the following inequality:
$$\sum_{|v|\leq x}\frac{\log |v|}{|v|}\leq \log x+ C_q,$$ where $v$'s are monic prime polynomials and $C_q$ is a constant regarding $q$. To see this, we notice that 
$$\sum_{|v|\leq x}\frac{\log |v|}{|v|}=\sum_{i=1}^{n:=\lfloor \log x \rfloor}\frac{i}{q^i}\cdot a_i,$$ where $a_i$ is the number of monic prime polynomials of degree $i$. By \cite[Theorem 2.2]{rosen2002} we obtain
$$\sum_{i=1}^{n}\frac{i}{q^i}\cdot a_i=\sum_{i=1}^n\left( 1+
O({q^{-i/2}})\right) \leq \log x+C_q. $$

Recall that $$e_{f_0}(v)=\frac{(1-\chi(v))(1-l^{-v(f_0)})}{(l-\chi(v))(1-l^{-1})}, \text{ where } l=|v|.$$
Note that $|l|\geq 3$ and $\chi(v)\in\{-1,0,1\}$. Thus we get 
$$e_{f_0}(v)\leq \frac{2}{l-l^{-1}}\leq \frac{9}{4l}.$$
Thus we have
$$\sum_{v|f_0}\deg(v)e_{f_0}(v)\leq \frac{9}{4}\sum_{v|f_0}\frac{\log l}{l}=\frac{9}{4}\left( \sum\limits_{\substack{v|f_0\\ |v|\leq \log |f_0|}}\frac{\log l}{l}+\sum\limits_{\substack{v|f_0\\ |v|>\log |f_0|}}\frac{\log l}{l}\right).$$
We have proven $$\sum\limits_{\substack{v|f_0\\ |v|\leq \log |f_0|}}\frac{\log l}{l} \leq \log\log |f_0|+C_q.$$ For the other term, we have 
$$\sum\limits_{\substack{v|f_0\\ |v|>\log |f_0|}}\frac{\log l}{l}\leq \frac{\log\log |f_0|}{\log|f_0|}\cdot \sum\limits_{\substack{v|f_0\\ |v|>\log |f_0|}} 1.$$
Using the product formula we see there are at most $\log |f_0|$ monic prime factors of $f_0$. As \cite[Theorem 2.2]{rosen2002} is effective, all the constant terms are summed up to a computable constant $C_q$.
\end{proof}

\begin{thm}\label{lbji}
Let $J$ be a singular modulus of rank 2 Drinfeld $A$-module with corresponding discriminant $\delta$ with conductor $f_0$. There exists some computable constant $C_q$ with respect to $q$ such that
$$h(J)\geq (q^2-1)\left( \frac{1}{2}-\frac{1}{\sqrt{q}+1}\right) \log\sqrt{|\delta|}+\left( \frac{1}{2}+\frac{1}{\sqrt{q}+1}\right)\log|f_0|-\frac{9}{4}\log\log |f_0|-C_q.$$
\end{thm}

\begin{proof}
We recall that for $K=k(\sqrt{\delta_0})$ with $\delta_0$ square free, the genus $g_K$ of $K$ is given by (cf. \cite[Section 3]{fb05})
\begin{equation*}
g_K=
\begin{cases}
\frac{\log|\delta_0|-1}{2} & \text{if $\deg(\delta_0)$ is odd,}\\
\frac{\log|\delta_0|-2}{2} & \text{if $\deg(\delta_0)$ is even.}
\end{cases}
\end{equation*}
 We input a result from \cite[Equation (23)]{bpr21}, which says that
 $$|h_G^{\infty}(\phi')-h_G^{\infty}(\phi)|\leq \frac{q}{q-1}-\frac{q^r}{q^r-1},$$
 where  $\phi$ and $\phi'$ are two isogenous Drinfeld $A$-modules of rank $r$. Using Corollary \ref{vghr2}, Lemma \ref{lbth}, Lemma \ref{etad} and the facts that $h_G(\phi)\geq \sth(\phi)$ and $h(J)=(q^2-1)h_G(\phi)$ we complete our proof.
\end{proof}

\begin{rem}
\begin{enumerate}[(1)]
\item We note that \cite[Equation (23)]{bpr21} holds true only for reduced Drinfeld modules \cite[definition before Lemma 4.2]{bpr21}. However, Lemma \ref{ivgh} ensures in the rank 2 case we can always choose the graded height of the Drinfeld module with associated lattice being the CM order, hence reduced.
\item One can also take $q\geq 3$ to make our statement independent of $q$, i.e.
$$h(J)\geq 4(2-\sqrt{3})\log \sqrt{|\delta|}+\frac{1}{2}\log |f_0|-\frac{9}{4}\log \log |f_0|-C_3.$$
\end{enumerate}
\end{rem}

Now we are ready to prove our main theorem. We restate our theorem:

\begin{thm}
Let $q$ be odd and $q>5$. There are at most finitely many singular moduli of rank 2 Drinfeld $\fq[t]$-modules that are algebraic units.
\end{thm}

\begin{proof}
Let $J$ be a singular modulus that is an algebraic unit. By Proposition \ref{ubji} and Theorem \ref{lbji} we have
\begin{equation}\label{uper}
	h(J)\leq \left(1+\frac{q^2-q}{d}\right)(q+1)\log \sqrt{|\delta|}+O_q(1)
\end{equation}
and
\begin{equation}\label{lowb}
	h(J)\geq (q^2-1)\left( \frac{1}{2}-\frac{1}{\sqrt{q}+1}\right) \log\sqrt{|\delta|}+\left( \frac{1}{2}+\frac{1}{\sqrt{q}+1}\right)\log|f_0|-\frac{9}{4}\log\log |f_0|-C_q,
\end{equation}
where $\delta$ is the discriminant of the endomorphism ring of a CM Drinfeld module whose $j$-invariant is $J$. By easy calculation and taking $d$ large enough if necessary, we find that for $q>5$
$$\left(1+\frac{q^2-q}{d}\right)(q+1)\log \sqrt{|\delta|} < (q^2-1)\left( \frac{1}{2}-\frac{1}{\sqrt{q}+1}\right) \log\sqrt{|\delta|}.$$
If we set 
$$\lambda_1:=\left(1+\frac{q^2-q}{d}\right)(q+1),\ \lambda_2:=(q^2-1)\left( \frac{1}{2}-\frac{1}{\sqrt{q}+1}\right),$$
then the above inequality tells us that $\lambda_2-\lambda_1>0$. Combining (\ref{uper}) and (\ref{lowb}) we get
$$(\lambda_2-\lambda_1)\log\sqrt{|\delta|}<O_q(1)+C_q-\left( \frac{1}{2}+\frac{1}{\sqrt{q}+1}\right)\log|f_0|+\frac{9}{4}\log\log |f_0|.$$
We note that $\left( \frac{1}{2}+\frac{1}{\sqrt{q}+1}\right)\log|f_0|-\frac{9}{4}\log\log |f_0|\geq 0$. Thus, for any such $\delta$ there exists a constant upper bound for $\log \sqrt{|\delta|}$. Lemma \ref{tsl} implies that this is also a constant upper bound for $h(z_i)$ for $z_i$. We note that $z_i$ has degree 2. Thus the Northcott theorem implies our theorem.
\end{proof}

\begin{appendices}
\section{More on Complex Multiplication}
We use notations as in section 2. This appendix is mainly devoted for the proof of Proposition \ref{abdf}. Actually the results stated here are already known for elliptic curves. The results below may be already known to many experts. Because of a lack of literature for Drinfeld modules, the details are worked out here for the convenience of the reader.\\

Let $\phi_1$ and $\phi_2$ both be rank $r$ Drinfeld $A$-module over $\Cinf$ with complex multiplication. Let $F/k$ be a finite field extension such that both $\phi_1$ and $\phi_2$ are defined over $F$ with everywhere good reduction. Let $R$ be the integral closure of $A$ in $F$ and denote $\mathscr{M}_1$, $\mathscr{M}_2$ the minimal model over $R$ of $\phi_1$, $\phi_2$ respectively. If $f:\mathscr{M}_1\ra \mathscr{M}_2$ is an isogeny, then it induces an isogeny of Drinfeld modules after taking reduction at a prime $ v \in \spec(R)$.

\begin{lem}\label{hinj}
Let $\phi_1^v$ and $\phi_2^v$ be Drinfeld modules over $\resk (v)$ obtained by taking reduction on $\mathscr{M}_1$ and $\mathscr{M}_2$ respectively at $v$, where $\resk(v)$ is the residue field at $v$.  Let $\Hom_F(\phi_1,\phi_2)$ denote the set of isogenies over $F$ between $\phi_1$ and $\phi_2$, similarly for $\Hom_{\resk(v)}(\phi_1^v,\phi_2^v)$. Then there is a canonical injection of $A$-modules:
$$\Hom_F(\phi_1,\phi_2)\xhookrightarrow{} \Hom_{\resk(v)}(\phi_1^v,\phi_2^v).$$
\end{lem}

\begin{proof}
It is easy to check the map is a morphism of $A$-modules. By Proposition \ref{gmp} and \cite[Proposition 2.5]{tag1993} we obtain
$$\Hom_F(\phi_1,\phi_2)=\Hom_R(\mathscr{M}_1,\mathscr{M}_2).$$
Let $f\in \Hom_R(\mathscr{M}_1,\mathscr{M}_2)$ be an isogeny. Then it is finite, which implies it has leading coefficient that is non-zero after reduction at $v$. This proves the injectivity.
\end{proof}

From now on, we assume $A=\fq[t]$ and $\phi$ is a rank $2$ Drinfeld $A$-module over $\Cinf$ with CM by the maximal order $\Ocal_K$, where $K\subset \Cinf$ is a quadratic imaginary field over $k$. Further we assume $\phi$ is obtained through the $A$-lattice $\Ocal_K$ and $\phi$ is defined over $F$. Let $P\in A$ be a prime element. 

\begin{thm}\label{cord}
For any place $v\in \spec(R)$, the reduction of $\phi$ at $v$ is ordinary if and only if $v\cap k$ splits in $\Ocal_K$.
\end{thm}

\begin{proof}
Let $P=v\cap k$ and $\bar{\pi}$ be the Frobenius morphism of $\phi^v$. From Lemma \ref{hinj} we see 
$$K=\End_F(\phi_1)\otimes_A k\xhookrightarrow{} \End_{\resk(v)}(\phi^v)\otimes_A k:=D.$$
By \cite[Proposition 4.12.17]{go98} we deduce that $\phi$ is ordinary at $v$ if and only if $K=\End_{\resk(v)}(\phi^v)\otimes_A k$. This is equivalent to saying that $\Ocal_K=\End_{\resk(v)}(\phi^v)$. We can embed $A$ into $\End_{\resk(v)}(\phi^v)$ via the Drinfeld module $\phi^v$. Let $E:=k(\bar{\pi})$. Then there is only one prime $\mathscr{P}$ of $E$ containing $\bar{\pi}$ and $\mathscr{P}$ lies over $P$ \cite[Theorem 4.12.8]{go98}.\\

If $\phi$ has ordinary reduction at $v$, then $E\cong K$ as $\bar{\pi}\notin k$. Again by \cite[Proposition 4.12.17]{go98} there are more than one primes of $E$ lying over $P$. Thus $P$ splits in $\Ocal_K$. Next we show the other way around. First, we write $P\Ocal_K=\mathcal{P}\mathcal{P}'$. Assume the reduction of $\phi$ at $v$ is supersingular. So it is a consequence that $\dim_kD=r^2=4$. Since $\phi$ has good reduction at $v$, $\phi^v$ has rank 2 over $\resk(v)$. Thus we have $2=\text{rank}(\phi^v)=t\cdot [E:k]$, where $t$ is an integer such that $t^2=\dim_E D$. As $\dim_kD=4=\dim_ED\cdot [E:k]$, we have $t=2$. Therefore, $E=k$. In particular, $\bar{\pi}\in A$. In this case, it is clear that $\mathscr{P}=(P)\subset A$. On the other hand, we can obtain a Drinfeld $\Ocal_K$-module $\psi$ over $F$ by extending $\phi$ to $\End_F(\phi)$. By taking reduction at $v$ again, we obtain a Drinfeld $\Ocal_K$-module $\psi^v$ over $\resk(v)$. It is trivial $\bar{\pi}$ is the Frobenius element of $\psi^v$. As $\bar{\pi}\in A\subset K$, there is only one prime ideal of $K$ containing $\bar{\pi}$. However, $\bar{\pi}\in P\Ocal_K=\mathcal{P}\mathcal{P}'$. This is a contradiction.
\end{proof}

\begin{cor}
If $P\Ocal_K=\mathscr{P}\mathscr{P}'$ where $\mathscr{P}$ and $\mathscr{P}'$ are different prime ideals of $\Ocal_K$ both lying over $P$, then for any place $v\in \spec(R)$ over $P$ the natural morphism $\End_F(\phi)\ra \End_{\resk(v)}(\phi^v)$ is an isomorphism.
\end{cor}

\noindent {\large \textbf{Reduction process}} \medskip

Let $\mathscr{M}$ be the minimal model of $\phi$. We set $\mathscr{M}[P]:=\text{Ker}(\phi_P:\phi\ra \phi)$. We suppose moreover that $P\Ocal_K=\mathscr{P}\mathscr{P}'$ with $\mathscr{P}$ and $\mathscr{P}'$ different. Then it is easy to see:
$$\mathscr{M}[P](\bar{A})=\mathscr{M}[P](\Cinf)\cong \Ocal_K/P\Ocal_K=\mathscr{P}/P\Ocal_K\oplus \mathscr{P}'/P\Ocal_K\cong \Ocal_K/\mathscr{P}\oplus \Ocal_K/\mathscr{P}'.$$
There is a natural morphism $\theta: \mathscr{M}[P](\bar{A})\ra \phi^v[P](\overline{\fq})$ by taking reduction at $v\in \spec(R)$ such that $v$ lies over $P$.

\begin{thm}\label{rdp}
If we assume further that $\pi\in \mathscr{P}\subset \Ocal_K=\End_F(\phi)=\End_R(\mathscr{M})$ is the lifting of the Frobenius element $\bar{\pi}$, then $\theta$ is a surjection, and the kernel of $\theta$ is isomorphic to $\Ocal_K/\mathscr{P}$.
\end{thm}

\begin{proof}
Since $P\Ocal_K=\mathscr{P}\mathscr{P}'$, by Theorem \ref{cord} $ \phi^v[P](\overline{\fq})$ is non-trivial and finite. We embed $A$ into $\End_{\resk(v)}(\phi^v)=\Ocal_K$ via $\phi^v$. As an $\Ocal_K$-module, we have $ \phi^v[P](\overline{\fq})\cong \Ocal_K/I$ for some proper ideal $I\subset \Ocal_K$. Therefore, we have $\phi^v_P\cdot \Ocal_K/I=0$, which implies $P\in I$. It is clear $$\#\{\phi^v[P](\overline{\fq})\}< \#\{\mathscr{M}[P](\Cinf)\}.$$ 
So either $I=\mathscr{P}$ or $I=\mathscr{P}'$. Since $\bar{\pi}$ acts on $\phi^v[P](\overline{\fq})$ non-trivially, we see $I=\mathscr{P}'$. Therefore, the kernel of $\theta$ is $\mathscr{P}'/P\Ocal_K$ that is isomorphic to $\Ocal_K/\mathscr{P}$.
\end{proof}

\begin{rem}
	\begin{enumerate}[(1)]
\item If we identify $\Ocal_K=\End_F(\phi)=\End_R(\mathscr{M})$, then $\mathscr{P}$ is the collection of isogenies whose reduction has linear coefficient 0.
\item Another approach to Theorem \ref{rdp} using canonical subgroup of Drinfeld modules has be shown to the author by Urs Hartl. The two approaches essentially have the same core. 
	\end{enumerate}
\end{rem}

\end{appendices}

\bibliographystyle{siam}
\bibliography{bibliography}

\begin{thebibliography}{10}

\bibitem{bm1940}
{\sc M.~Becker and S.~Maclane}, {\em The minimum number of generators for
  inseparable algebraic extensions}, Bulletin of the American Mathematical
  Society, 46 (1940), pp.~182--186.

\bibitem{bg2006}
{\sc E.~Bombieri and W.~Gubler}, {\em Heights in Diophantine Geometry}, New
  Mathematical Monographs, Cambridge University Press, 2006.

\bibitem{breuer02}
{\sc F.~Breuer}, {\em Sur la conjecture d'Andr\'e-Oort et courbes modulaires de
  Drinfeld}, PhD thesis, Université Denis Diderot, 2002.

\bibitem{fb05}
{\sc F.~Breuer}, {\em The {A}ndr\'e-{O}ort conjecture for products of
  {D}rinfeld modular curves}, Journal für die reine und angewandte Mathematik,
  2005 (2005), pp.~115--144.

\bibitem{bpr21}
{\sc F.~Breuer, F.~Pazuki, and M.~H. Razafinjatovo}, {\em Heights and isogenies
  of {D}rinfeld modules}, Acta Arithmetica, 197 (2021), pp.~111 -- 128.

\bibitem{brown92}
{\sc M.~Brown}, {\em Singular moduli and supersingular moduli of {D}rinfeld
  modules.}, Inventiones Mathematicae, 110 (1992), pp.~419--440.

\bibitem{bhk2020}
{\sc Y.~Buli, P.~Habegger, and L.~K\"uhne}, {\em No singular modulus is a
  unit}, International Mathematics Research Notice, 2020 (2020),
  pp.~10005--10041.

\bibitem{cu04}
{\sc L.~Clozel and E.~Ullmo}, {\em \'Equidistribution des points de Hecke},
  Contributions to automorphic forms, geometry, and number theory, John Hopkins
  Univ. Press, Baltimore, MD, 2004, pp.~193--254.

\bibitem{gek99}
{\sc E.-U. Gekeler}, {\em {A survey on Drinfeld modular forms}}, Turkish
  Journal of Mathematics, 23 (1999), pp.~485--518.

\bibitem{go98}
{\sc D.~Goss}, {\em Basic Structures of Function Field Arithmetic},
  Springer-Verlag Berlin Heidelberg, 1998.

\bibitem{ha14}
{\sc P.~Habegger}, {\em {Singular moduli that are algebraic units}}, Algebra
  and Number Theory, 9 (2015), pp.~1515 -- 1524.

\bibitem{dh1979}
{\sc D.~Hayes}, {\em Explicit class field theory in global function fields},
  Studies in algebra and number theory, Academic Press, New York-London, 1979,
  pp.~173--217.

\bibitem{ih2006}
{\sc Y.~Ihara}, {\em On the Euler-Kronecker constants of global fields and
  primes with small norms}, Algebraic Geometry and Number Theory: In Honor of
  Vladimir Drinfeld's 50th Birthday, Birkh{\"a}user Boston, Boston, MA, 2006,
  pp.~407--451.

\bibitem{ill1985}
{\sc L.~Illusie}, {\em D\'eformations de groupes de {B}arsotti-{T}ate},
  Ast\'erisque, 127 (1985), pp.~151 -- 198.

\bibitem{leh09}
{\sc T.~Lehmkuhl}, {\em Compactification of the Drinfeld Modular Surfaces},
  vol.~197, Memoirs of the American Mathematical Society, 2009.

\bibitem{li2021}
{\sc Y.~Li}, {\em Singular units and isogenies between {CM} elliptic curves},
  Compositio Mathematica, 157 (2021), p.~1022–1035.

\bibitem{nt1991}
{\sc Y.~Nakkajima and Y.~Taguchi}, {\em A generalization of the
  {C}howla-{S}elberg formula}, Journal für die reine und angewandte
  Mathematik, 419 (1991), pp.~119 -- 124.

\bibitem{ray1974}
{\sc M.~Raynaud}, {\em Sch\'emas en groupes de type (p,...,p)}, Bulletin de la
  Soci\'et\'e Math\'ematique de France, 102 (1974), pp.~241--280.

\bibitem{ray1985}
\leavevmode\vrule height 2pt depth -1.6pt width 23pt, {\em Hauteurs et
  isog\'enies}, Ast\'erisque, 127 (1985), pp.~199 -- 234.

\bibitem{rosen2002}
{\sc M.~Rosen}, {\em Number Theory in Function Fields}, Graduate Texts in
  Mathematics, Springer New York, 2002.

\bibitem{tag1991}
{\sc Y.~Taguchi}, {\em Semisimplicity of the {G}alois represebtations attached
  to {D}rinfeld modules over fields of ``finite characteristics"}, Duke
  Mathematical Journal, 62 (1991), pp.~593--599.

\bibitem{tag1993}
\leavevmode\vrule height 2pt depth -1.6pt width 23pt, {\em Semi-simplicity of
  {G}alois represebtations attached to {D}rinfeld modules over fields of
  ``infinite characteristics"}, Journal of Number Theory, 44 (1993),
  pp.~292--314.

\bibitem{wei17}
{\sc F.-T. Wei}, {\em Kronecker limit formula over global function fields},
  American Journal of Mathematics, 139 (2017), pp.~1047 -- 1084.

\bibitem{wei20}
{\sc F.-T. Wei}, {\em {On Kronecker terms over global function fields}},
  Inventiones Mathematicae, 220 (2020), pp.~847--907.

\end{thebibliography}

\Addresses

\end{document}